\documentclass[11pt]{amsart}
\usepackage[all]{xy}
\usepackage{amsmath}
\usepackage{amssymb}
\usepackage{amsthm}
\usepackage{verbatim}
\newtheorem{theorem}{Theorem}[section]
\newtheorem{lemma}[theorem]{Lemma}

\numberwithin{equation}{section}

\numberwithin{equation}{section}
\newtheorem{proposition}[theorem]{Proposition}

\newfont{\EUL}{eufm10 scaled 1000}
\newcommand{\be}{\begin{equation}}
\newcommand{\ee}{\end{equation}}
\newcommand{\bl}{\begin{lemma}}
\newcommand{\el}{\end{lemma}}
\newcommand{\bp}{\begin{proof}}
\newcommand{\ep}{\end{proof}}

\newcommand\R{\mathbb{R}}

\newcommand\C{\mathbb{C}}

\renewcommand\H{\mathbb{H}}

\newcommand\Ker{{\rm Ker}}

\newcommand\su{\mathfrak{su}}
\renewcommand\sp{\mathfrak{sp}}

\renewcommand\sl{\mbox{\EUL sl}}

\newcommand\g{\mathfrak{g}}
\newcommand\h{\mathfrak{h}}
\newcommand\z{\mathfrak{z}}
\renewcommand\u{\mathfrak{u}}

\newcommand\m{\mathfrak{m}}
\newcommand\n{\mathfrak{n}}
\renewcommand\t{\mathfrak{t}}
\renewcommand\c{\mathfrak{c}}
\renewcommand\l{\mathfrak{l}}
\newcommand\f{\mathfrak{f}}
\newcommand\s{\mathfrak{s}}

\newcommand\ttt{{\tilde{\theta}}}

\renewcommand\span{{\rm span}\,}

\newcommand\ad{{\rm ad}}

\newcommand\Ad{{\rm Ad}}

\def\sideremark#1{\ifvmode\leavevmode\fi\vadjust{
\vbox to0pt{\hbox to 0pt{\hskip\hsize\hskip1em
\vbox{\hsize1cm\tiny\raggedright\pretolerance10000
\noindent #1\hfill}\hss}\vbox to8pt{\vfil}\vss}}}

\begin{document}
%
%
\title[Homogeneous Hypercomplex structures and the Joyce's construction]{Homogeneous
Hypercomplex structures and the Joyce's construction}
\author{Lucio Bedulli - Anna Gori - Fabio Podest\`a}
\address{Dipartimento di Matematica
- Sapienza, Universit\`a di Roma\\
p.le Aldo Moro 2\\00185 Roma\\Italy\\{ \it e-mail address}: {\tt bedulli@mat.uniroma1.it}}
\address{Dipartimento di Matematica "U.Dini"
- Universit\`a di Firenze\\
V.le Morgagni 67/A\\50100 Firenze\\Italy\\{ \it e-mail address}: {\tt gori@math.unifi.it}}
\address{Dipartimento di Matematica "U.Dini"
- Universit\`a di Firenze\\
V.le Morgagni 67/A\\50100 Firenze\\Italy\\{ \it e-mail address}: {\tt podesta@math.unifi.it}}
\date{}
\begin{abstract} We prove that any invariant hypercomplex structure on
  a homogeneous space $M = G/L$ where $G$ is a compact Lie group is
  obtained via the Joyce's construction, provided that
  there exists a hyper-Hermitian naturally reductive invariant metric on $M$.  \end{abstract}
\maketitle
\section{Introduction}
A hypercomplex structure $\mathcal Q$ on a manifold $M$ is a set of
integrable complex structures on $M$ of the form $\mathcal Q = \{ aI +
bJ + cK;\ a^2 + b^2 + c^2 = 1\}$, where $I$, $J$, $K=IJ$ are complex structures satisfying $IJK = -1$.
A Riemannian metric $g$ on $M$ is called hyper-Hermitian if it is
Hermitian w.r.t. every complex structure in $\mathcal Q$;
it is easy to see that hyper-Hermitian metrics always exist.
A hyper-Hermitian manifold ($M,g,\mathcal Q$) admits a HKT-structure,
where HKT means hyper-K\"ahler with torsion, when there exists a
metric connection $\nabla$ that leaves every complex structure in $\mathcal Q$ parallel and whose
torsion tensor $T$ is totally skew.
When a HKT-structure exists, the connection $\nabla$ is unique and it
is called the HKT-connection; actually it coincides with the Bismut
connection of every complex structure in $\mathcal Q$ (see e.g. \cite{Gau}). We refer the reader to \cite{GP,V1,V2} 
for equivalent definitions and basic properties of HKT-structures,
which have also played an important role in theoretical physics (see e.g. \cite{HP1, HP2}).\par
Hyperk\"ahler structures are a special case of HKT-structures, namely
when the HKT-connection coincides with the Levi-Civita connection of
$g$, i.e. it is torsionfree.
Actually, the hyperk\"ahler condition is extremely stringent and
examples are rare; for instance homogeneity forces the manifold to be
flat (see e.g. \cite{Besse}). On the contrary, there are plenty of examples of HKT-structures,
even when $M$ is supposed to be compact and homogeneous. In
\cite{SSTV} the authors described and classified all the left
invariant hypercomplex structures on compact Lie groups, for which
there exists a biinvariant, hyper-Hermitian Riemannian metric. Joyce
\cite{Joyce} then described a way to construct  hypercomplex
structures on homogeneous spaces of compact Lie groups; this
construction, which we recall in section \ref{Jcon}, has been then
used and revisited by several authors, see e.g. \cite{OP,PP}. Our
first main result states that, if $G$ is a compact Lie group, then
every $G$-invariant hypercomplex structure on a homogeneous $G$-space
is obtained via the Joyce construction, provided that there exists a
naturally reductive $G$-invariant, hyper-Hermitian metric;
this metric automatically endows the homogeneous space with an invariant HKT-structure. As a corollary of the proof of this first statement, we also get the fact that the semisimplicity of $G$ forces the group to be of a special kind, namely with every simple factor of type $A_n$. These results are summarized in the following
\begin{theorem}\label{MainThm}
Let $G$ be a compact connected Lie
  group acting transitively and almost effectively on some manifold $M$ preserving a
  hypercomplex structure $\mathcal Q$. Suppose that there exists a
  naturally reductive $G$-invariant metric
$g$ on $M$ which is hyper-Hermitian w.r.t. $\mathcal Q$.
Then \par
\begin{enumerate}
\item there exists a Cartan subalgebra of the complex reductive algebra $\g^\C$ and a corresponding root system for the semisimple part of $\g^\C$ such that the hypercomplex structure $\mathcal Q$ coincides with the one given by the Joyce's construction;
\item if $G$ is semisimple, then every simple factor of $\g$ is of type $A_n$.

\end{enumerate}
\end{theorem}
The existence of a naturally reductive metric which is hyper-Hermitian
is supposed and extensively used in \cite{SSTV} as well as in our
arguments, while we are unaware of any result proving it. We refer the
reader also to \cite[Theorem 4]{OP}, where the existence of a family of
invariant HKT structures on compact homogeneous spaces is discussed.
Our last result reduces the existence of a hyper-Hermitian naturally reductive metric on a homogeneous space to the case of a
Lie group.
\begin{proposition}\label{restr} Let $G$ be a compact Lie group and $M = G/L$ a $G$-homogeneous space endowed with an invariant hypercomplex structure $\mathcal Q$. Suppose $L$ is not trivial and connected. Then
\begin{enumerate}
\item The connected component $Y$ of the fixed point set $M^L$ of $L$ containing the
  origin $[eL]$ is a positive dimensional hypercomplex submanifold. In particular $\chi(M) = 0$;
\item If $g$ is an invariant naturally reductive metric, it is hyper-Hermitian if and only if its restriction to $Y$ is.\end{enumerate}
\end{proposition}

{\sc Notation}. Lie groups and their Lie algebras will be indicated by
capital and gothic letters respectively. We will denote the Cartan-Killing form by $\mathcal{B}$.

\section{Preliminaries}
\subsection{Invariant complex structures}
\label{ics}
In order to establish the notation we briefly recall the structure theory of
compact homogeneous complex manifolds; the reader is referred e.g. to \cite{Ak} for a more
detailed exposition.\par
Let $M$ be a compact complex manifold and let $G$ be a
compact connected Lie group acting almost effectively, transitively and holomorphically on $M$. We
will write $M = G/L$ for some compact subgroup $L$. Up to a finite
covering we can assume that $L$ is connected. We will also denote by $I$ the $G$-invariant complex structure
on $M$. \par
The complexified group $G^\C$ acts holomorphically on $G/L$, so that $M=G^\C/Q$ for
some connected complex subgroup $Q \subset G^\C$.
It is well known that the {\em Tits fibration} $\phi$ provides a
holomorphic fibering of the homogeneous space $M$ over a compact rational homogeneous space $G^\C/P$, where the
parabolic subgroup $P$ is in general defined to be the normalizer
$N_{G^\C}(Q^o)$ of $Q^o$ (see \cite{Ak}); since $Q$ is connected the fibres of $\phi$ are
 complex tori. The flag manifold $G^\C/P$ can be written as
$G/C$ endowed with a $G$-invariant complex structure $\mathcal{I}$,
where $C$ is the centralizer of some torus in $G$.
Accordingly the Lie algebra $\g$ can be decomposed as
\begin{equation}\label{dec}
\g=\l\oplus\t\oplus\n\,,\end{equation}
where $\c=\l\oplus\t$ and $\n$ is an $\Ad(C)$-invariant complement of
$\c$ in $\g$. Since $\t$ identifies with the tangent space to the fibre we have
$[\t,\t]=0$.
Moreover the algebra $\c$ is contained in the normalizer of $\l$ in $\g$
by construction, hence $[\l,\t]\subset\l \cap \t = 0$.
We can choose a Cartan subalgebra $\h$ of the form $\h=\t_\l^\C\oplus
\t^\C$, where $\t_\l$ is a maximal abelian subalgebra of $\l$.
Denote by $R$ the corresponding root system of $\g^\C$, by $R_\l$ the subsystem relative to $\l$ so that
$\l^\C = \t_\l \oplus \bigoplus_{\alpha\in R_\l}\g_\alpha$, and by $R_\n$ the
symmetric subset of $R$ such that $\n^\C=\bigoplus_{\alpha \in R_\n}\g_\alpha$. The $G$-invariant complex structure
$\mathcal I$ induces an endomorphism of $\n^\C$ that is $\Ad(C)$-invariant and therefore the corresponding subspace
$\n^{1,0}$ is a sum of root spaces. The integrability of $\mathcal I$ is equivalent to the condition
$$[\n^{1,0},\n^{1,0}]_{\n^\C} \subseteq \n^{1,0}$$
and one can prove (see e.g. \cite{BFR}) that there is a suitable ordering of $R_\n=R_\n^+ \cup R_\n^-$ such that
$$\n^{1,0} = \bigoplus_{\alpha\in R_\n^+} \g_\alpha, \quad
\n^{0,1} = \bigoplus_{\alpha\in R_\n^-} \g_\alpha.$$
The $G$-invariant complex structure $I$ on $G/L$ induces an $\Ad(L)$-invariant endomorphism, still denoted by $I$,
of $\m^\C$, where $\m := \t + \n$. It leaves both $\t$ and $\n$ invariant with $I|_\n = \mathcal I$ and
the integrability of $I$ is equivalent to the vanishing of the Nijenhuis tensor $N_I$, namely  for $X,Y\in \m$
\begin{equation}\label{NI}
[IX,IY]_\m - [X,Y]_\m - I[IX,Y]_\m - I[X,IY]_\m = 0.
\end{equation}
Equation \eqref{NI} is trivial for $X,Y\in \t$ and with $X\in \t$ and
$Y\in \n$ it reduces to the $\ad(\t)$-invariance of $I$. When $X,Y\in \n$, then \eqref{NI} is the integrability of $\mathcal
I$ because $[\n^{1,0},\n^{1,0}]\subseteq \n^{1,0}$. \par
Viceversa, we start with a decomposition as in
\eqref{dec}, where $\l+\t = \c$ and $\c$ is the centralizer of an abelian
subalgebra.
If we fix an $\ad(\c)$-invariant integrable complex
structure $\mathcal I$ on $\n$ and we extend it by choosing an arbitrary complex structure $I_\t$ on $\t$,
then $I_\t + \mathcal I$ will provide an integrable $L$-invariant complex structure on the homogeneous space $G/L$.

\subsection{Hypercomplex and HKT structures}
A {\em hypercomplex} structure on a manifold $M$ is determined by a pair of
anticommuting complex structures. Whenever such a pair $(I,J)$ is
given, one has a 2-sphere of complex structures on $M$ given by
$\{aI + bJ + c IJ \colon a,b,c \in \R \quad {\rm and} \quad
a^2+b^2+c^2=1\}$. \\
A metric $g$ on a hypercomplex manifold $M$ is called
{\em hyper-Hermitian} if it is Hermitian with respect to both the
complex structures.
A metric connection $\nabla$ on a hyper-Hermitian manifold $(M,g,I,J)$ is called
{\em hyper-K\"ahler with torsion} (HKT) if $\nabla I =\nabla J
=0$ and the torsion $T^\nabla$ is totally skew-symmetric, i.e. the
tensor $\tau(X,Y,Z)=g(T^\nabla(X,Y),Z)$ is a 3-form.\\
Note that if such a connection exists, then it is unique since it is
the {\em Bismut connection} of each complex structure (see
\cite{Gau}).  \\
An important technical tool we use in section \ref{Main} is the
following well-known fact (see \cite{YM}).
\begin{proposition}\label{PNmisto}
Let $(M,I)$ be a complex manifold. An almost complex structure $J$ on
$M$ anticommuting with $I$ is integrable if and only if the tensor
$N_{IJ}$ defined for $X,Y \in \Gamma(TM)$ as
\begin{equation}
\label{Nmisto}
N_{IJ}(X,Y) = [IX,JY]+[JX,IY]-I[JX,Y]-J[IX,Y]-I[X,JY]-J[X,IY]
\end{equation}
vanishes identically on $M$.
\end{proposition}

Using the notation of subsection \ref{ics} we consider a homogeneous space
$M=G/L$ where the compact group $G$ preserves a hypercomplex structure
generated by $I, J$ and $K=IJ$. \\
Given a $G$-invariant decomposition $\g = \l + \m$ and the
corresponding {\em canonical connection} $D$, it is known that any $G$-invariant tensor on
the homogeneous space $G/L$ is $D$-parallel (see e.g. \cite{KoNo}). Since the torsion
of $D$ is given by $T^D(X,Y)= - [X,Y]_\m$, we see that $D$ becomes the HKT-connection if there exists a
$G$-invariant naturally reductive metric on $M$ that is Hermitian
with respect to $I$ and $J$.
\subsection{The Joyce construction}\label{Jcon}
In \cite{Joyce} Joyce explains how to construct invariant hypercomplex
structures on suitable compact homogeneous spaces. His construction
can be outlined as follows. Given a compact Lie algebra $\g$ we fix a
Cartan subalgebra $\h$ of $\g^\C$ and denote by $R$ the set of
corresponding roots. We can find a sequence $\theta_1,\ldots,\theta_k$
of roots such that if $\mathfrak{s}_i^\C \cong \sl_2(\C)$ is the
subalgebra generated by the root spaces $\g_{\theta_i},
\g_{-\theta_i}$ and
\[
\mathfrak{f}_i:=\g \cap\bigoplus_{{\mathcal B}(\alpha,\theta_i)\neq0}\!\g_{\alpha}\;\,\,, \quad
\mathfrak{b}_i:= \g \cap \bigcap_{j=1}^i \,\z_{\g^\C}(\mathfrak{s}_j^\C)\;\,,
\]
where $\z_{\g^\C}(\s_j^\C)$ denotes the centralizer of $\s_j^\C$ in $\g^\C$,
then one has the decomposition
\begin{equation}
\label{Jd}
\g=\mathfrak{b}_k \oplus \bigoplus_{i=1}^k \s_i \oplus \bigoplus_{i=1}^k \f_i\,.
\end{equation}

The Lie algebra of the isotropy $\l\subset\mathfrak{b}_k$ is chosen as
follows: the semisimple part of $\l$ coincides with the semisimple
part of $\mathfrak{b}_k$ and the center $\z_\l$ of $\l$ is a subset of
the center $\z'$ of $\mathfrak{b}_k$ such that $\dim \z' -\dim \z_\l
-k \equiv_4 0$. We denote by $\m$ the $\mathcal B$-orthogonal complement of
$\l$ in $\g$.\par
The invariant hypercomplex structure on $G/L$ is obtained by the
following $\Ad(L)$-invariant hypercomplex structure $\mathcal Q$ on
$\m$.
The structure $\mathcal Q_{|\f_i}$ coincides with $\ad(\s_i)$.
We select $\mathcal B$-orthogonal vectors $u_1,\ldots,u_k$ in $\z'\cap \m$
and use the fact that $\mathfrak{q}_i=\s_i\oplus \R u_i\cong \H$ to
define $\mathcal Q_{|\mathfrak{q}_i}$. The complement of $\z_\l\oplus
\sum_i \R u_i$ in $\z'$ can be endowed with an arbitrary linear
hypercomplex structure.

\section{Proof of the main results}\label{Main}
\subsection{Proof of Theorem \ref{MainThm}, part (1)} We write $M = G/L$ for some closed subgroup $L\subset G$. We will
also suppose that $L$ is connected, otherwise we pass to a finite covering.
We will suppose that $G/L$ admits a naturally
reductive metric $g$ with respect to the reductive decomposition $\g =
\l + \m$, which is Hermitian w.r.t. every complex structure in
$\mathcal Q$.
We recall (see e.g. \cite{KoNo}) that the metric $g$ induces a scalar product on $\m$ such that, for every $X,Y,Z\in \m$
\begin{equation}
\label{red} g([X,Y]_\m,Z) + g(Y,[X,Z]_\m) = 0.
\end{equation}
Using \eqref{red} and the $\Ad(L)$-invariance of $g$, it is immediate
to see that $g(\t,\n)=0$.
\par
We fix one complex structure $I\in \mathcal Q$ and apply the structure theory explained in section 2.1, keeping the same notation. If now $J\in \mathcal Q$ is an integrable complex structure anticommuting with $I$, we think of $J$ as an $\Ad(L)$-invariant endomorphism of $\m$ and we may formulate
Proposition \ref{PNmisto} as follows: for every $X,Y\in \m$
\begin{equation}\label{misto}
[IX,JY]_\m+[JX,IY]_\m-I[JX,Y]_\m-J[IX,Y]_\m-I[X,JY]_\m-J[X,IY]_\m = 0.
\end{equation}
If we now extend $J$ to the complexification $\m^\C$, we see that $J(\m^{1,0}) = \m^{0,1}$ and
$J(\m^{0,1}) = \m^{1,0}$. If $X\in \m^{1,0}$ and $Y\in \m^{0,1}$, equation (\ref{misto}) reduces to
$$i[X,JY]_{\m^\C} -i[JX,Y]_{\m^\C} - I[JX,Y]_{\m^\C} - I[X,JY]_{\m^\C} = 0, $$
which is automatically satisfied since
$[\m^{1,0},\m^{1,0}]_{\m^\C}\subseteq \m^{1,0}$.\\
If $X,Y\in \m^{1,0}$ we see that $N_{IJ}(X,Y)\in \m^{0,1}$ and therefore equation (\ref{misto}) is equivalent to
$g(N_{IJ}(X,Y),Z) = 0$ for every  $Z\in \m^{1,0}$. Using the fact that $g$ is naturally reductive and Hermitian
w.r.t. $I$ and $J$, we have that equation (\ref{misto}) is equivalent to the following condition: for every
$X,Y,Z\in \m^{1,0}$ the cyclic sum
\begin{equation}
\label{cyclic}
\mathfrak{S}_{{\small (X,Y,Z)}} \,\, g(JX,[Y,Z]_{\m^\C}) = 0.
\end{equation}
We now consider the root system $R$ associated with the choice of the maximal abelian subalgebra $(\t_\l+\t)^\C$ of
$\g^\C$ as described in section 2.1. The root subsystem $R_\n$ where $\m = \t \oplus \n$ has an ordering $R_\n =
R_\n^+ \cup R_\n^-$ induced by the complex structure $I$ and we can
select a root $\theta\in R_\n^+$ which is maximal w.r.t. this
ordering, namely for every $\alpha\in R_\n^+$ $$\theta + \alpha
\not\in R_\n^+.$$
Throughout the following we will denote by $\{H_\alpha,E_\alpha\}_{\alpha \in
  R}$ the standard Chevalley's basis of the semisimple part of
$\g^\C$.\\
We here remark that, since the metric $g$
is naturally reductive,  for $\alpha,\beta\in R_\n^+$ we have
$g(E_\alpha,E_\beta)=0$ whenever $\alpha\neq -\beta$ and
$g(E_\alpha,E_{-\alpha}) \neq 0$. Moreover if $iH_\alpha \in \t$ one
can see that
$g(E_\alpha,E_{-\alpha})= -\frac{\|\alpha\|^2}{|\alpha|^2}$, where $\|\alpha\|^2=g(iH_\alpha,iH_\alpha)$ .

\begin{lemma} \label{Et}We have $JE_{\theta}\in \t^\C$. In particular $E_\theta$ is centralized by $\l$. \label{lemma0}
\end{lemma}
\begin{proof}
Since $g(\t,\n) = 0$, we need to show that $g(JE_{\theta},E_\alpha)=0$ whenever
$\alpha \in R_\n^+$.
To do this we can first take $X=E_{\theta}$, $Y=E_\alpha$ and $Z= H \in \t^\C \cap
\m^{(1,0)}$ in formula \eqref{cyclic} and obtain
\[
(\alpha+\theta)(H)\ g(JE_{\theta},E_\alpha)=0\,.
\]
Now, if $\alpha + \theta$ does not vanish on $\t^\C \cap\m^{(1,0)}$,
then the claim follows. Otherwise $\alpha + \theta$ vanishes on the whole $\t^\C$ since $\alpha
+ \theta \in i\t^*$; in this case we can take $H' \in \t_\l$ such that $(\alpha +
\theta)(H') \neq 0$. For such a $H'$ we have
$[H',JE_{\theta}]=\theta(H')JE_{\theta}$ and, contracting with
$E_\alpha$ and using the fact that $g$ is $J$-Hermitian, once again we get
\[
(\alpha+\theta)(H')\ g(JE_{\theta},E_\alpha)=0\,,
\]
obtaining our first claim. The second assertion follows from the fact that $[\l,\t]=0$ and the
$\ad(\l)$-invariance of $J$.
\end{proof}

Now we want to compute the $\t^\C$-component of $JE_\alpha$ for $\alpha\in R_\n^+$.
\begin{lemma}\label{t-component}
Given $\alpha\in R_\n^+$ the following statements hold:
\begin{itemize}
\item[(i)] If $\alpha|_{\t_\l} \equiv 0$ then $JE_\alpha =
k_\alpha(H_\alpha+iIH_\alpha)\, {\rm{mod}} \, \n^\C$ for some $k_\alpha
\in \C$. In particular $JE_\theta =
k_\theta(H_\theta+iIH_\theta)$, where $|k_\theta|^2 = \frac{1}{2|\theta|^2}$.
\item[(ii)] If $\alpha|_{\t_\l} \not\equiv 0$ then $JE_\alpha \in \n^\C.$
\end{itemize}
\end{lemma}
\begin{proof} In order to prove (i) we first note that $iH_\alpha$
  lies in $\t$. We now apply \eqref{cyclic} taking
$X=E_\alpha$, $Y=H_1$ and $Z=H_2$ where
$H_1,H_2\in \t^\C\cap \m^{1,0}$. Thus we obtain
\[g(JE_\alpha,\alpha(H_2)H_1-\alpha(H_1)H_2)=0.
\]
The linear space $\span_\C\{\alpha(H_2)H_1-\alpha(H_1)H_2:\:H_1,H_2\:\in
\t^\C\cap \m^{1,0}\}$ coincides with $\{v-iIv:\:v,Iv\:\in (\Ker\ \alpha)\cap \t\}.$
This means that the $\t^\C$-component of $JE_\alpha$ is of the form
$\gamma (w+iIw)$ with $\gamma \in \C$ and $w \in \t$ is $g$-orthogonal
to $\Ker\, \alpha$. Since $g(iH_\alpha, \Ker\,\alpha)=0$ we can choose $w=iH_\alpha$
and the claim follows for a suitable $k_\alpha \in\C$.
The last assertion follows from the following computation
\begin{eqnarray*} g(E_\theta,E_{-\theta}) &=& |k_\theta|^2 g(H_\theta - iIH_\theta, H_\theta + i IH_\theta) =
2|k_\theta|^2 g(H_\theta,H_\theta) = \\
&=& 2|k_\theta|^2 g([E_\theta,E_{-\theta}],H_\theta) = 2|k_\theta|^2 |\theta|^2 g(E_\theta,E_{-\theta}).\end{eqnarray*}
As for (ii), we select $H\in \t_\l$ with $\alpha(H)\neq 0$ and
use the $\ad(\l)$-invariance of $J$ to compute
$$\alpha(H)\ g(JE_\alpha,\t) = g([H,JE_\alpha],\t) = g(JE_\alpha,[H,\t]) = 0.$$
\end{proof}
We note that $k_\theta$ is determined up to multiplication by a complex number of unit norm, since $J$ can be
chosen in the circle of complex structures in $\mathcal Q$ which are orthogonal to $I$.

\begin{lemma}
\label{lista}
\begin{itemize}
\item[(i)] If $\alpha, \beta \in R_\n^+$ and $\alpha +\beta \not \in
  R$, then $g(JE_\alpha,E_\beta)=0$.
\item[(ii)] If $\alpha, \beta \in R_\n^+$ and $\alpha +\beta = \gamma
  \in R^+$ with $\gamma_{|{\t_\l}} \not\equiv 0$, then $g(JE_\alpha,E_\beta)=0$.
\item[(iii)] If $\alpha, \beta \in R_\n^+$ and $\alpha +\beta = \gamma
  \in R^+$ with $\gamma_{|\t_\l}\equiv 0$, then $g(JE_\alpha,E_\beta)=
  2 k_{\gamma} \frac{\|\gamma\|^2}{|\gamma|^2}N_{\alpha,\beta}$.
\end{itemize}
\end{lemma}
\begin{proof}
The first assertion can be easily proved with the same argument used
in Lemma \ref{lemma0}. In order to prove (ii), let $H\in \t_\l$ with $\gamma(H)\neq 0$ and use the
$\ad(\l)$-invariance of $J$ to compute
\begin{eqnarray*}\alpha(H)g(JE_\alpha,E_\beta) & = & g(J[H,E_\alpha],E_\beta)=
g([H,JE_\alpha],E_\beta)=-g(JE_\alpha,[H,E_\beta])\\
 & =& -g(JE_\alpha,\beta(H)E_\beta)=
-\beta(H)g(JE_\alpha,E_\beta)
\end{eqnarray*}
so that the claim follows.\\
As for (iii), we select $H\in \t$ with $\gamma(H)\neq 0$ and set $H' = H - iIH$. Using \eqref{cyclic} we have
\[
\gamma(H')g(JE_\alpha,E_\beta)=g(JH',[E_\alpha,E_\beta])=N_{\alpha,\beta}\: g(JH',E_\gamma)
=-N_{\alpha,\beta}\: g(JE_\gamma,H').
\]
Applying part (i) of the previous Lemma
we get
\begin{eqnarray*}
\gamma(H')g(JE_\alpha,E_\beta) & = & -N_{\alpha,\beta}\:
k_\gamma(g(H_\gamma,H')+ig(IH_\gamma,H')) =
-2N_{\alpha,\beta}k_\gamma\, g(H_\gamma,H') \\
 & = & 2\,\gamma(H')\,N_{\alpha,\beta}\,k_\gamma\, \frac{\|\gamma\|^2}{|\gamma|^2}\,.
\end{eqnarray*}
and the claim follows.
\end{proof}
We now consider the highest root $\theta$ and define $R(\theta) = \{\alpha\in R_\n^+;\ \theta-\alpha\in R\}$.
Note that $\alpha\in R_\n^+$ lies in $R(\theta)$ if and only if $\alpha\neq \theta$ and $\mathcal B(\alpha,\theta) \neq 0$. Moreover
 if $\alpha \in R(\theta)$, then $\theta-\alpha\in R_\n$: indeed, if
 $\theta-\alpha = \beta\in R_\l$, we have
$\theta-\beta = \alpha\in R$, hence $[E_\theta,E_{-\beta}]\neq 0$, contradicting the fact that $[\l,E_\theta] = 0$ (see
Lemma \ref{Et}).

\bl
\label{JE}  If $\alpha\in R(\theta)$, then $JE_\alpha \in \n^\C$.
\el

\bp  Suppose $JE_\alpha$ has a component along $\t^\C$. Using Lemma
\ref{t-component}, we compute
$$0 = g(JE_\alpha, JE_{-\theta}) =  k_\alpha k_\theta\ g(H_\alpha + iI H_\alpha, H_\theta + i IH_\theta) =
2 k_\alpha k_\theta\ g(H_\alpha,H_\theta) =    $$
$$ = 2 k_\alpha k_\theta\ g([E_\alpha,E_{-\alpha}],H_\theta) = -2 k_\alpha k_\theta\ \alpha(H_\theta)
g(E_\alpha, E_{-\alpha}).$$
Since $\alpha\in R(\theta)$ we have that $\alpha(H_\theta)\neq 0$. Therefore $k_\alpha = 0$ and the claim follows.
\ep

\bl
If $\alpha\in R(\theta)$, then $g(JE_\alpha,E_\beta) = 0$ for every $\beta\in R_\n^+$ unless $\alpha + \beta = \theta$.
\el

\bp
By Lemma \ref{lista} (i), it is enough to take $\beta\in R_\n^+$ so that $\alpha+\beta = \gamma\in R$. Moreover by Lemma
\ref{lista} (iii), we may suppose that $\gamma|_{\t_\l}\equiv 0$, hence $\gamma|_{\t}\not\equiv 0$. Choose $H\in \t^{1,0}$
with $\gamma(H) \neq 0$. Now, if $\gamma\neq \theta$, by Lemma \ref{JE} we have that $g(J[E_\alpha,E_\beta],\t) = 0$.
Equation (\ref{cyclic}) with $X = E_\alpha$, $Y = E_\beta$ and $Z= H$ implies
$\gamma(H)\ g(JE_\alpha,E_\beta) = 0$ and the claim follows.
\ep

The previous Lemma says that for every $\alpha \in R_\n^+$ one has $JE_\alpha=\lambda_\alpha E_{\alpha-\theta}$ for
some $\lambda_\alpha \in \C\setminus\{0\}$. Using Lemma \ref{lista}
(iii) we have
\begin{equation}
\label{lambda}
\lambda_\alpha g(E_{\alpha-\theta},E_{\theta-\alpha}) =
g(JE_\alpha,E_{\theta-\alpha}) = 2 k_\theta \frac{\|\theta\|^2}{|\theta|^2}
N_{\alpha,\theta-\alpha}\,.
\end{equation}
Using the fact that $g$ is naturally reductive we have
\[
g(E_{\alpha-\theta},E_{\theta-\alpha})=
-\frac{N_{\alpha,\theta-\alpha}}{N_{\alpha,-\theta}}g(E_{-\theta},E_\theta)
= \frac{N_{\alpha,\theta-\alpha}}{N_{\alpha,-\theta}}\frac{\|\theta\|^2}{|\theta|^2}\,,
\]
which, combined with \eqref{lambda} gives $$JE_\alpha = 2 k_\theta
N_{\alpha,-\theta}E_{\alpha-\theta}\,.$$

Let $\mathfrak{s}(\theta)^\C$ be the subalgebra of $\g^\C$ generated by
$E_\theta$ and $E_{-\theta}$, and define $\mathfrak{s}(\theta)=\mathfrak{s}(\theta)^\C\cap
\g$. Obviously $\mathfrak{s}(\theta)\cong\sp(1)$. Set also
$$Z_\theta = I(i H_\theta)\in \t,\quad \mathfrak{u}(\theta) = \mathfrak{s}(\theta) \oplus \R\ Z_\theta.$$
Then $\mathcal Q$ leaves $\mathfrak{u}(\theta)$ invariant and $\mathcal{Q}_{|\u(\theta)}$ is
determined by the formula $JE_\theta = k_\theta (H_\theta +iIH_\theta)$.
We also define $\mathfrak{f}_\theta=\g \cap \bigoplus_{\alpha\in R(\theta)}
(\g_\alpha \oplus \g_{-\alpha})$ and $\mathfrak{c}_\theta = \g \cap \bigoplus_{\alpha\in C(\theta)}
(\g_\alpha \oplus \g_{-\alpha})$, where $C(\theta) = \{\alpha\in R_\n^+;\ (\theta,\alpha) = 0\} =
R_\n^+\setminus (R(\theta)\cup \{\theta\})$,  so that
$$\n \oplus \span_\R\{iH_\theta, Z_\theta\} = \mathfrak{u}(\theta) \oplus \mathfrak{f}_\theta \oplus \mathfrak{c}_\theta.$$
\begin{proposition}
\label{ad}
The hypercomplex structure
$\mathcal{Q}$ leaves $\mathfrak{f}_\theta$ invariant and
$\mathcal{Q}_{|\mathfrak{f}_\theta} = \ad(\mathfrak{s}(\theta))_{|\mathfrak{f}_\theta}$.
\end{proposition}
\bp
We will show that there exist $\sigma_\theta, \tau_\theta \in
\s(\theta)$ such that for every $X \in \mathfrak{f}_\theta$ we have $JX=
[\sigma_\theta,X]$ and $IX=[\tau_\theta,X]$.
Let $\sigma_\theta = 2(\overline{k}_\theta E_\theta - k_\theta
E_{-\theta})$ and $\tau_\theta = \frac{2iH_\theta}{|\theta|^2}$.
The claim is a consequence of the following direct computations
\begin{eqnarray*}
[\sigma_\theta,E_\alpha]  & = &  -2k_\theta[E_{-\theta},E_\alpha] = -2
k_\theta N_{-\theta,\alpha}E_{\alpha-\theta} = JE_\alpha \\
{[\tau_\theta,E_\alpha]} & = & \frac{2i}{|\theta|^2}[H_\theta,E_\alpha]
                       = 2\frac{{\mathcal B}(\alpha,\theta)}{|\theta|^2}iE_\alpha
                       = iE_\alpha \, \\
\end{eqnarray*}
where in the last equation we have used the fact that
$2\frac{{\mathcal B}(\alpha,\theta)}{(\theta,\theta)}=1$ since the
$\theta$-string of $\alpha$ is formed only by $\alpha-\theta$ and
$\alpha$ (see e.g. \cite{Helgason}).
\ep

We now set $\theta_1:=\theta, \, k_1:=k_\theta$ and define inductively
the roots $\theta_j$ as follows.
\begin{itemize}
\item[1)] $\theta_{j+1}$ is maximal in $C(\theta_j)$,
  i. e. $\theta_{j+1}+\alpha \not \in R$ for every $\alpha \in
  C(\theta_j)$;
\item[2)] $C(\theta_{j+1}):=\{\alpha \in C(\theta_j)\colon
  \theta_{j+1}-\alpha \not \in R \}$
\end{itemize}
 We then set $R(\theta_{j+1}) = \{\alpha\in C(\theta_j);\ \theta_{j+1}
 - \alpha \in R\}$ and $\f_{j+1} = \g \cap \bigoplus_{\alpha\in R(\theta_{j+1})}
(\g_\alpha \oplus \g_{-\alpha})$. Moreover we define
$\s_{j+1} \cong \sp(1)$ as the real subalgebra generated by
$E_{\theta_{j+1}}, E_{-\theta_{j+1}}$ (note that $\s_1 = \s(\theta)$) and $\u_{j+1} =
\s_{j+1} \oplus \R Z_{j+1}$ where $Z_{j+1} =
iIH_{\theta_{j+1}} \in \t$. \\
Now we have
\begin{proposition}\label{Joyce}
There exists a set of roots $\theta_1,\ldots,\theta_\ell$ such that for $j=1,\ldots,\ell$ we have:
\begin{itemize}
\item[(i)]  the subset $C(\theta_\ell)$ is empty;
\item[(ii)] the hypercomplex structure $\mathcal{Q}$ leaves
$\f_j$ and $\u_j$ invariant. In particular
$\mathcal{Q}_{|\f_j}=\ad(\s_j)_{|\f_j}$ and we have
$JE_{\theta_j}=k_j (H_{\theta_j} + iIH_{\theta_j}) $
for a suitable $k_j \in \C$  (hence $\l$
centralizes $\s_j$);
\item[(iii)] there is a $g$-orthogonal decomposition
$\g = \l \oplus \tilde\t \oplus \bigoplus_{j=1}^{\ell} \u_j \oplus \bigoplus_{j=1}^{\ell} \f_j$, where
$\tilde\t$ lies in $\t$ and is $\mathcal Q$-invariant. Moreover $[\l,\u_j] = 0$, $[\u_j,\u_k]=0$ for $j\neq k$ and
$[\u_j,\f_j]\subseteq \f_j$;
\item[(iv)] the root $\theta_1$ can be chosen as the
highest root $\tilde{\theta}$ of the whole root system $R$ of $\g$
with respect to an ordering such that $R^+ \supseteq R^+_\n$.
\end{itemize}
\end{proposition}
\begin{proof}
The first three statements can be proved by induction using exactly the same
arguments as in the previous Lemmas and in Proposition \ref{ad}.
The only new statement to prove is (iv).
To do this it is enough to show that the highest root space $\g_\ttt$
does not belong to $\l^\C$. Suppose now by contradiction that $E_\ttt \in \l^\C$.
Given $\alpha\in R_\n^-$, we have $JE_{\alpha} = H + \sum_{\beta\in R_\n^+}c_\beta E_{\beta}$ for some
$H\in \t^\C$ and $c_\beta\in \C$ and therefore $[E_\ttt,E_\alpha]= -J[E_\ttt, JE_\alpha] = 0$ because
$\ttt+R_\n^+\not\subset R$ and $[\l,\t]=0$.
 Hence $[E_\ttt,\n]= 0$ and therefore $[E_\ttt,\m^\C]= 0$. But this cannot
happen otherwise $\exp_G(E_\ttt-E_{-\ttt})$ would act trivially on
$M$, contradicting the (almost) effectiveness of the $G$-action.\end{proof}

Note that the decomposition obtained above matches with decomposition
\eqref{Jd} if we take $\mathfrak{b}_k = \l \oplus \tilde\t \oplus \z$,
where $\z$ is the center of $\bigoplus_{j=1}^{\ell} \u_j$.
We also note that we have the following necessary condition: if
$\z_\ell$ is the center of the centralizer in $\g$ of
$\{s(\theta_1),\ldots,s(\theta_\ell)\}$, then
\begin{equation} \label{cnec}\dim \z_\ell \geq \ell.\end{equation}

\subsection{Proof of Theorem \ref{MainThm}, part (2)}
We first prove the claim in the case in which $G$ is simple, using Proposition \ref{Joyce} and
condition \eqref{cnec}. \par
Since the root $\theta_1$ can be chosen as the
highest root $\tilde{\theta}$ of the whole root system, we can start
from the ``Wolf decomposition''
of $\g$ with respect $\theta_1=\ttt$:
\[
\g = \s(\theta_1) \oplus \z_\g(\s(\theta_1)) \oplus \m_1
\]
where $\m_1$ is identified with the tangent space of the corresponding
Wolf space.\\
By a case-by-case inspection for simple groups it is not difficult to see that
for every set of strongly orthogonal roots $\theta_1=\ttt,\ldots,\theta_\ell$
of $\g$, we have $\dim\z_\ell < \ell$ unless
$\g$ is of type $A_n$. If $\g=\su(n)$ we have indeed $\dim \z_\ell =
\ell$ for every choice of  $\theta_1=\ttt,\ldots,\theta_\ell$. (see
also \cite[Proposition 1]{PP}).
\\
Suppose now that $\g=\g_1\oplus\ldots\oplus\g_r$ where the $\g_j$'s
are simple Lie algebras. The set of roots $\Theta = \{\theta_1,\ldots,\theta_\ell\}$ is the disjoint
union of the subsets $\Theta_j$ of all roots in $\Theta$ belonging to $\g_j$. Now $\z_\ell$ splits as
a direct sum of the centers $\z_j$ of the centralizers in $\g_j$ of the subalgebras generated by the roots in
$\Theta_j$. Then $\dim \z_\ell = \sum_{j=1}^r \dim \z_j < \sum_{j=1}^r \sharp\Theta_j = \ell$ if at least one
factor of $\g$ is not of type $A_n$ by the previous discussion.

\subsection{Proof of Proposition \ref{restr}} (1) Suppose that $Y$ is reduced to a point. For any $I\in \mathcal Q$ the
corresponding Tits fibration $\pi$ has a typical fiber that is pointwisely fixed by the isotropy $L$, hence trivial.
This means that $M$ is a flag manifold with an invariant hypercomplex structure. If we decompose $\g = \l + \m$ with
$\m$ an $\ad(\l)$-invariant subspace, it is known that the $\ad(\l)$-irreducible submodules $\m_j$ ($j=1,\ldots,k$)
of $\m$ are mutually inequivalent (see e.g. \cite{S}) and therefore $\mathcal Q$-invariant. Now $\l$ has a non trivial
center $\c$ and there is a submodule, say $\m_1$, such that $\ad(\c)|_{\m_1}$ is not trivial. Then using the
irreducibility of $\m_1$, we see that $\mathcal Q|_{\m_1}$ belongs to $\ad(\c)|_{\m_1}$, contradicting the fact that
the $\mathcal Q|_{\m_1}$ contains anti-commuting elements. Therefore $Y$ has positive dimension and is $\mathcal Q$-invariant.
Since $L$ is not trivial, we see that $Y$ is also a proper submanifold. \par
(2) Suppose now that the restriction of $g$ to $Y$ is hyper-Hermitian and consider the decomposition
$\g = \l + \t + \n$ as in section \ref{ics}, relative to some $I\in \mathcal Q$. Note that $[\l,\t]=0$
means that $\t$ projects to a subspace of $T_{[eL]}Y$ and therefore $g|_{\t\times\t}$ is $I$-Hermitian. Now $\n^\C$
is a sum of root spaces w.r.t. the Cartan subalgebra $(\t_\l + \t)^\C$ and a simple computation using the natural
reductiveness and the $\ad(\l)$-invariance of $g$ shows that $g(E_\alpha,E_\beta) = 0$ for every roots $\alpha,\beta$
with $\alpha+\beta\neq 0$. Our claim now follows form the fact that $g(IE_\alpha,IE_{-\alpha}) = g(E_\alpha,E_{-\alpha})$
for every root $\alpha$.


\begin{thebibliography}{9999}
%
\bibitem{Ak}{\sc D. Akhiezer}, Lie Group Actions in Complex Analysis, Aspects in Math. vol E27 Vieweg (1995)

\bibitem{BFR}{\sc M.Bordemann, M. Forger, H. R\"omer}, {\em Homogeneous K\"ahler manifolds: paving the way toward new supersymmetric sigma models}, Comm. Math. Phys. vol {\bf 102} (1986) 605--647
%
\bibitem{Besse}{\sc A.L. Besse}, Einstein manifolds, {\em Ergebnisse
  der Mathematik und ihrer Grenzgebiete}, 10. Springer-Verlag, Berlin,
1987.
%
\bibitem{S} {\sc De Siebenthal J.}, {\em Sur certains modules dans une alg\`ebre de Lie
semisimple}, Comment. Math. Helv. {\bf 44} (1964) 1--44
%
\bibitem{Gau}{\sc P. Gauduchon}, {\em Hermitian connections and Dirac
    operators}, Boll. Un. Mat. Ital. B (7)  {\bf 11}  (1997),  no. 2, suppl., 257--288.
%
\bibitem{GP}{\sc G. Grantcharov, Y.S. Poon}, {\em Geometry of Hyper-K\"ahler Connections with Torsion},
Comm. Math. Phys. {\bf 213} (2000), 19--37
%
\bibitem{Helgason}
{\sc S. Helgason:} Differential Geometry, Lie Groups and Symmetric Spaces,
New York-London: Academic Press-inc (1978)
%
\bibitem{HP1} {\sc P.S. Howe, G. Papadopoulos}, {\em Further remarks on the geometry of two-dimensional nonlinear $\sigma$ models},  Classical Quantum Gravity  {\bf 5} (1988), 1647--1661.
%
\bibitem{HP2} {\sc P.S. Howe, G. Papadopoulos}, {\em Twistor spaces for hyper-K\"ahler manifolds with torsion},  Phys. Lett. B  {\bf 379}  (1996), 80--86.

%
\bibitem{Joyce}{\sc D. Joyce}, {\em Compact hypercomplex and
    quaternionic manifolds}, J. Differential Geom.  {\bf 35} (1992), 743--761.
%
\bibitem{KoNo}{\sc S. Kobayashi, K. Nomizu}, Foundations of
differential geometry. Interscience Tracts in
 Pure and Applied Mathematics, No. 15 Vol. II   John Wiley Sons, Inc., New York-London-Sydney 1969
%
%
\bibitem{OP} {\sc A. Offermann, G. Papadopoulos}, {\em Homogeneous HKT and QKT manifolds}, arXiv:math-ph/9807026v1 (1998).
%
\bibitem{PP} {\sc H. Pedersen, Y.-S. Poon}, {\em
Inhomogeneous hypercomplex structures on homogeneous manifolds},
J. Reine Angew. Math. {\bf 516} (1999), 159--181.
%
\bibitem{SSTV}
{\sc Ph. Spindel, A. Sevrin, W. Troost,  A. Van Proeyen}, {\em Extended supersymmetric $\sigma$-
models on group manifolds}, Nucl. Phys. B308 (1988) 662--698
%
\bibitem{V1} {\sc M. Verbitsky}, {\em Hypercomplex manifolds with trivial canonical bundle and their holonomy},
 Moscow Seminar on Mathematical Physics. II,
Amer. Math. Soc. Transl. Ser. 2, {\bf 221} (2007),  203--211.
%
\bibitem{V2} {\sc M. Verbitsky}, {\em Hypercomplex structures on K\"ahler manifolds},  Geom. Funct. Anal.
{\bf 15}  (2005), 1275--1283.

%
\bibitem{YM} {\sc K. Yano, M. Ako}, {\em Integrability conditions for
    almost quaternion structures}, Hokkaido Math. J., {\bf 1}  (1972), 63--86.




%
%
%
%
%
%
%
%
%

%
%
%
%
%
%
%
%
%
%

\end{thebibliography}
\end{document}